\documentclass{amsart}

\usepackage{amsmath, amssymb, amsthm, amsrefs}
\usepackage{enumerate}
\usepackage{lineno}
\usepackage{tikz}
\usepackage{hyperref}
\usepackage[center]{caption}
\usepackage[b]{esvect}

\theoremstyle{plain}
\newtheorem{theorem}{Theorem}[section]
\newtheorem{lemma}[theorem]{Lemma}
\newtheorem{proposition}[theorem]{Proposition}
\newtheorem{corollary}[theorem]{Corollary}

\theoremstyle{definition}
\newtheorem{definition}[theorem]{Definition}

\newcommand{\N}{\mathbb{N}}
\newcommand{\Z}{\mathbb{Z}}
\newcommand{\R}{\mathbb{R}}
\newcommand{\ol}{\overline}
\newcommand{\eps}{\varepsilon}
\newcommand{\vp}{\varphi}

\DeclareMathOperator{\pag}{pag}
\DeclareMathOperator{\val}{val}
\DeclareMathOperator{\cp}{\square}

\begin{document}

\title{Peg solitaire and Conway's soldiers on infinite graphs}

\author{Valentino Vito}
\address{Faculty of Computer Science\\
Universitas Indonesia\\
Depok 16424, Indonesia}
\email{valentino.vito11@ui.ac.id}

\keywords{Peg solitaire, Conway's soldiers, infinite game, pagoda function, golden ratio}
\subjclass{Primary 05C57; Secondary 05C63, 91A43, 91A46}

\begin{abstract}
Peg solitaire is classically a one-player game played on a grid board containing pegs. The goal of the game is to have a single peg remaining on the board by sequentially jumping with a peg over an adjacent peg onto an empty cell while eliminating the jumped peg. Conway's soldiers is a related game played on $\mathbb{Z}^2$ with pegs initially located on the half-space $y \le 0$. The goal is to bring a peg as far up as possible on the board using peg solitaire jumps. Conway showed that bringing a peg to the line $y = 5$ is impossible with finitely many jumps. Applying Conway's approach, we prove an analogous impossibility property on graphs. In addition, we generalize peg solitaire on finite graphs as introduced by Beeler and Hoilman (2011) to an infinite game played on countable graphs.
\end{abstract}

\maketitle

\section{Introduction}\label{sec1}

Peg solitaire is traditionally known as a one-player game played on a grid board whose cells can each contain at most one peg. The grid board in its entirety may assume various shapes, most notably a plus-sign shape in the English version of peg solitaire. Every turn, the player jumps with a peg over an adjacent peg onto an empty cell in the same direction two cells away, eliminating the jumped peg from the board. The goal of the game is to eliminate every peg except for one. This game has been analyzed for both grids with square-shaped cells \cites{bell2007diagonal, bell2007fresh, jefferson2006} as well as grids with hexagonal-shaped cells \cites{bell2008, hentzel1968}. The reader may consult \cite{beasley1985} for an in-depth treatment of peg solitaire.

A well-known game associated with peg solitaire is Conway's soldiers, also known as \emph{solitaire army}. The game is played on the infinite grid $\Z^2$ with an initial configuration of pegs (or in the context of this game, soldiers) located everywhere on the integer lattice points of the half-space $y \le 0$. The goal of the game is to move a peg as far up as possible on the board with only peg solitaire jumps. A generalization of Conway's soldiers played on $\Z^d$ was considered by Eriksson and Lindstr\"om \cite{eriksson1995}. Additionally, the related notion of pegging numbers on graphs was studied in \cites{helleloid2009, levavi2013}. The \textit{pegging number} of a graph is defined as the minimum number of pegs so that no matter how the pegs are positioned on the graph, any vertex can be filled with a peg after some number of moves.

Conway \cite{berlekamp2004} proved that while we can bring a peg to the line $y = 4$, reaching $y = 5$ is impossible with finitely many jumps. Conway's argument relies on a so-called \emph{pagoda function} on the grid, where a peg is given a higher value as it moves further away from the initial half-space $y \le 0$. Conway showed that the total value of pegs on the board cannot increase throughout the game and that a peg located on $y = 5$ would produce a value greater than the total value of pegs at the start of the game, producing a contradiction. Conway's pagoda function is often known as the \emph{golden pagoda} for its reliance on the golden ratio.

In the present paper, we analyze peg solitaire and its variant Conway's soldiers on graph-based boards. Peg solitaire played on graphs was first proposed by Beeler and Hoilman \cite{beeler2011} in 2011. Since its introduction, studies on peg solitaire have been strictly done on finite graphs \cites{beeler2012, beeler2015, engbers2015}, whereas our focus here lies on infinite graphs. We refer to \cites{fraenkel2012} for a survey on combinatorial games in general.

\begin{figure}
    \centering
    \begin{tikzpicture}[x=1cm, y=1cm]
        \draw[thick] (0, 0) -- (2, 0);
        \draw[thick] (4, 0) -- (6, 0);
        \draw[dashed, ->] (0, 0.2) to [bend left=30] (2, 0.2);
        \draw[ultra thick, ->] (2.7, 0) -- (3.3, 0);
        
        \draw[fill=black] (0, 0) circle (2pt);
        \draw[fill=black] (1, 0) circle (2pt);
        \draw[fill=white] (2, 0) circle (2pt);
        \draw[fill=white] (4, 0) circle (2pt);
        \draw[fill=white] (5, 0) circle (2pt);
        \draw[fill=black] (6, 0) circle (2pt);
        
        \node at (0, -0.3) {$u$};
        \node at (1, -0.3) {$v$};
        \node at (2, -0.3) {$w$};
        \node at (4, -0.3) {$u$};
        \node at (5, -0.3) {$v$};
        \node at (6, -0.3) {$w$};
    \end{tikzpicture}
    
    \caption{The jump $u \cdot \protect\vv{v} \cdot w$.}
    \label{fig1}
\end{figure}
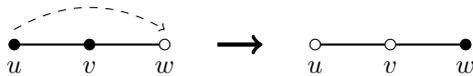

Our game of peg solitaire is played on a possibly infinite graph $G$ whose vertex and edge set are denoted by $V(G)$ and $E(G)$, respectively. This paper focuses on connected, countable graphs $G$, hence both $V(G)$ and $E(G)$ are assumed to be countable. Let $S_0 \subseteq V(G)$ be the \emph{initial state} of the game, which consists of vertices which contain a peg at the beginning of the game. Vertices not in $S_0$ are thus empty at game start. Each turn, the player jumps with a peg on a vertex $u$ over an adjacent vertex $v$ containing a peg into an empty vertex $w$ adjacent to $v$ other than $u$, emptying $v$ in the process. We denote this jump by $u \cdot \vv{v} \cdot w$, and we say that $u$, $v$, and $w$ are \textit{involved} in the jump; see Figure \ref{fig1} for illustration. After this first jump $j_0$, the state of the game changes from $S_0$ to $S_1$, where $S_1 = (S_0 \setminus \{u, v\}) \cup \{w\}$. In general, the state changes from $S_{n-1}$ to $S_n$ after the $n$-th jump $j_{n-1}$. The sequence $(S_n)_{n=0}^k$ of \emph{game states} then describes a game of peg solitaire in which jumps $(j_n)_{n=0}^{k-1}$ are made.

In Section \ref{sec2}, we present our version of Conway's soldiers on graphs, which we feel captures the spirit of the game well. The \textit{limit state} $S_\omega$ is defined in Section \ref{sec3}, and the notion of valued graphs is introduced in Section \ref{sec4} based on the discussion in Section \ref{sec2}. In Section \ref{sec5}, we study graphs whose pegs can be entirely cleared after infinitely many moves. Lastly, we provide several conditions for the clearability of a Cartesian product of graphs in Section \ref{sec6}.

\section{Soldiers on a graph}\label{sec2}

Recall that a graph is \emph{locally finite} if each of its vertices has finite degree. Let $G$ be a connected, locally finite graph, from which we pick some vertex $v$. Throughout this section, we designate $v$ as the \textit{goal point} of $G$. Let $\N = \{0, 1, 2, \dots\}$. For $n \in \N$, define $D_n(v)$ as the set of vertices of distance $n$ from $v$. This set is finite by local finiteness of $G$. Also, set $d_n(v) = |D_n(v)|$. The sequence $(d_n(v))_{n = 0}^\infty$ has no zero term whenever $G$ is an infinite connected graph; otherwise, the sequence eventually vanishes.

We say that the goal point $v$ is \emph{always reachable} if for every $k \in \N$, there exists an initial state $S_0 \subseteq \bigcup_{n=k}^\infty D_n(v)$ on $G$ such that a peg can be brought to $v$ with finitely many jumps; otherwise, $v$ is \emph{eventually unreachable}. Obviously, every vertex of a finite graph is eventually unreachable since $\bigcup_{n=k}^\infty D_n(v) = \emptyset$ for sufficiently large $k$. Our goal for this section is to present a relationship between the reachability of $v$ and the growth rate of $(d_n(v))$.

We denote the positive root of the polynomial $x^2 - x - 1$ by $\vp$. This constant is the famed \emph{golden ratio} and is equal to $\frac{1 + \sqrt{5}}{2}$. The reciprocal of $\vp$ is denoted by $\sigma$, and it satisfies the identity
\begin{equation}\label{eq1}
    \sigma^i = \sigma^{i+1} + \sigma^{i+2}.
\end{equation}
Note that $\sigma$ is equal to $\frac{\sqrt{5} - 1}{2}$. Our main theorem of this section is the following.

\begin{theorem}\label{thm1}
Let $G$ be a connected, locally finite graph, and let $v$ be its goal point. The vertex $v$ is eventually unreachable if $d_n(v) = O(\vp^{\eps n})$ for some $\eps < 1$. Moreover, this bound is sharp in that it does not necessarily hold if $\eps = 1$ is taken.
\end{theorem}

To prove Theorem \ref{thm1}, we need to define a suitable pagoda function on $G$. Recall from Section \ref{sec1} that a pagoda provides some value to each peg depending on its position on the board.

\begin{definition}
Let $G$ be a connected, locally finite graph. A \emph{pagoda on $G$} is a function $\pag \colon V(G) \to \R$ such that for every $3$-vertex path $abc$, we have
\begin{equation}\label{eq2}
    \pag(a) \le \pag(b) + \pag(c).
\end{equation}
The \emph{value of a state $S \subseteq V(G)$} is then defined as
\[\val(S) = \sum_{u \in S} \pag(u)\]
whenever the sum converges absolutely.
\end{definition}

A pagoda ensures that, given the convergence of $\val(S_0)$, the values of all other game states also converge and $\val(S_{n+1}) \le \val(S_n)$ for $n \in \N$; that is, the state value of the game does not increase throughout play. We set our pagoda on $G$ such that it depends on the chosen goal point $v$. The pagoda provides a higher value to vertices closer to $v$ than to vertices further away from $v$. Namely, we define
\[\pag(u) = \sigma^{d(u,v)}.\]
Using Equation (1), it is easy to see that this pagoda satisfies inequality (\ref{eq2}) and that equality holds when
\[d(a,v) + 2 = d(b,v) + 1 = d(c,v).\]

With this pagoda, we provide a condition on $(d_n(v))$ for the existence of $\val(S)$ for all $S$.

\begin{proposition}\label{prop1}
Let $G = (V, E)$ be a connected, locally finite graph, and let $v$ be its goal point. If $d_n(v) = O(\vp^{\eps n})$ for some $\eps < 1$, then $\val(S) = \sum_{u \in S} \sigma^{d(u,v)}$ converges for every state $S \subseteq V$.
\end{proposition}

\begin{proof}
The convergence of
\[\val(V) = \sum_{u \in V} \sigma^{d(u,v)} = \sum_{n = 0}^{\infty} d_n(v)\sigma^n\]
can be obtained from the convergence of
\[\sum_{n = 0}^{\infty} \vp^{\eps n}\sigma^n = \sum_{n = 0}^{\infty} \vp^{(\eps-1)n}\]
for $\eps < 1$. It follows that $\val(S) = \sum_{u \in S} \sigma^{d(u,v)}$ converges for all $S \subseteq V$.
\end{proof}

\begin{figure}
    \centering
    \begin{tikzpicture}[x=3.2cm, y=3.2cm]
        \draw[thick] (0, 2) -- (0, 1/8);
        \draw[thick] (1, 1) -- (1, 1/8);
        \draw[thick] (3/2, 1/2) -- (3/2, 1/8);
        \draw[thick] (1/2, 1/2) -- (1/2, 1/8);
        \draw[thick] (7/4, 1/4) -- (7/4, 1/8);
        \draw[thick] (5/4, 1/4) -- (5/4, 1/8);
        \draw[thick] (3/4, 1/4) -- (3/4, 1/8);
        \draw[thick] (1/4, 1/4) -- (1/4, 1/8);
        \draw[thick] (1, 1) -- (0, 1);
        \draw[thick] (3/2, 1/2) -- (1, 1/2);
        \draw[thick] (1/2, 1/2) -- (0, 1/2);
        \draw[thick] (7/4, 1/4) -- (3/2, 1/4);
        \draw[thick] (5/4, 1/4) -- (1, 1/4);
        \draw[thick] (3/4, 1/4) -- (1/2, 1/4);
        \draw[thick] (1/4, 1/4) -- (0, 1/4);
        \draw[thick] (15/8, 1/8) -- (7/4, 1/8);
        \draw[thick] (13/8, 1/8) -- (3/2, 1/8);
        \draw[thick] (11/8, 1/8) -- (5/4, 1/8);
        \draw[thick] (9/8, 1/8) -- (1, 1/8);
        \draw[thick] (7/8, 1/8) -- (3/4, 1/8);
        \draw[thick] (5/8, 1/8) -- (1/2, 1/8);
        \draw[thick] (3/8, 1/8) -- (1/4, 1/8);
        \draw[thick] (1/8, 1/8) -- (0, 1/8);
        
        \draw[fill=green] (0, 1) circle (4pt);
        \draw[fill=green] (0, 1/2) circle (4pt);
        \draw[fill=green] (0, 1/4) circle (4pt);
        \draw[fill=green] (0, 1/8) circle (4pt);
        \draw[fill=green] (1, 1/2) circle (4pt);
        \draw[fill=green] (1, 1/4) circle (4pt);
        \draw[fill=green] (1, 1/8) circle (4pt);
        \draw[fill=green] (3/2, 1/4) circle (4pt);
        \draw[fill=green] (3/2, 1/8) circle (4pt);
        \draw[fill=green] (1/2, 1/4) circle (4pt);
        \draw[fill=green] (1/2, 1/8) circle (4pt);
        \draw[fill=green] (7/4, 1/8) circle (4pt);
        \draw[fill=green] (5/4, 1/8) circle (4pt);
        \draw[fill=green] (3/4, 1/8) circle (4pt);
        \draw[fill=green] (1/4, 1/8) circle (4pt);
        
        \draw[fill=black] (0, 2) circle (2pt);
        \draw[fill=black] (0, 1) circle (2pt);
        \draw[fill=black] (0, 1/2) circle (2pt);
        \draw[fill=black] (0, 1/4) circle (2pt);
        \draw[fill=black] (0, 1/8) circle (2pt);
        \draw[fill=black] (1, 1) circle (2pt);
        \draw[fill=black] (1, 1/2) circle (2pt);
        \draw[fill=black] (1, 1/4) circle (2pt);
        \draw[fill=black] (1, 1/8) circle (2pt);
        \draw[fill=black] (3/2, 1/2) circle (2pt);
        \draw[fill=black] (3/2, 1/4) circle (2pt);
        \draw[fill=black] (3/2, 1/8) circle (2pt);
        \draw[fill=black] (1/2, 1/2) circle (2pt);
        \draw[fill=black] (1/2, 1/4) circle (2pt);
        \draw[fill=black] (1/2, 1/8) circle (2pt);
        \draw[fill=black] (7/4, 1/4) circle (2pt);
        \draw[fill=black] (7/4, 1/8) circle (2pt);
        \draw[fill=black] (5/4, 1/4) circle (2pt);
        \draw[fill=black] (5/4, 1/8) circle (2pt);
        \draw[fill=black] (3/4, 1/4) circle (2pt);
        \draw[fill=black] (3/4, 1/8) circle (2pt);
        \draw[fill=black] (1/4, 1/4) circle (2pt);
        \draw[fill=black] (1/4, 1/8) circle (2pt);
        \draw[fill=black] (15/8, 1/8) circle (2pt);
        \draw[fill=black] (13/8, 1/8) circle (2pt);
        \draw[fill=black] (11/8, 1/8) circle (2pt);
        \draw[fill=black] (9/8, 1/8) circle (2pt);
        \draw[fill=black] (7/8, 1/8) circle (2pt);
        \draw[fill=black] (5/8, 1/8) circle (2pt);
        \draw[fill=black] (3/8, 1/8) circle (2pt);
        \draw[fill=black] (1/8, 1/8) circle (2pt);
        
        \draw[fill=black] (1, -0.1) circle (1pt);
        \draw[fill=black] (1, -0.2) circle (1pt);
        \draw[fill=black] (1, -0.3) circle (1pt);
        \draw[fill=black] (-0.35, -0.1) circle (1pt);
        \draw[fill=black] (-0.35, -0.2) circle (1pt);
        \draw[fill=black] (-0.35, -0.3) circle (1pt);
        
        \node at (0.1, 2) {$v$};
        \node at (-0.35, 2) {layer $0$};
        \node at (-0.35, 1) {layer $1$};
        \node at (-0.35, 1/2) {layer $2$};
        \node at (-0.35, 1/4) {layer $3$};
        \node at (-0.35, 1/8) {layer $4$};
        
    \end{tikzpicture}
    
    \caption{A tree $T$ whose vertex $v$ is always reachable. We mark some vertices by a green circle, while the others are left unmarked.}
    \label{fig2}
\end{figure}
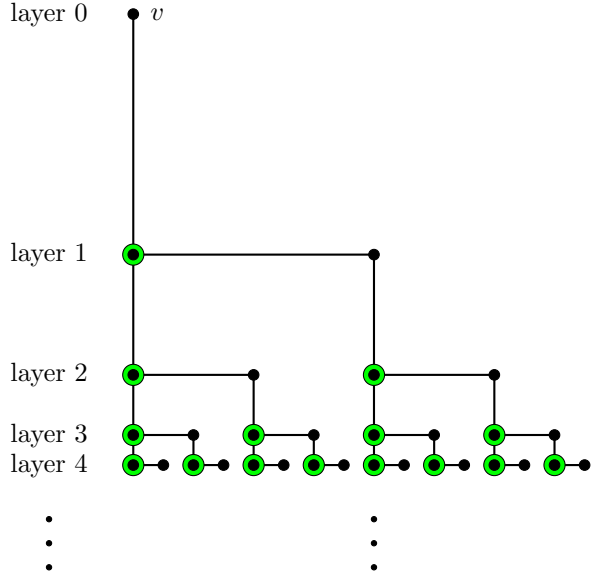

\begin{proof}[Proof of Theorem \ref{thm1}]
Let $S$ be any state with $v \in S$. Since $\sum_{n = 0}^{\infty} d_n(v)\sigma^n$ converges when $d_n(v) = O(\vp^{\eps n})$, we have
\[\sum_{n = k}^{\infty} d_n(v)\sigma^n < 1 \le \val(S)\]
for sufficiently large $k$. Hence for any initial state $S_0$ with $S_0 \subseteq \bigcup_{n=k}^\infty D_n(v)$, we cannot have a peg on $v$ during the game since value does not increase.

Now we show that our upper bound is best possible. Consider the infinite tree $T$ in Figure \ref{fig2} whose vertices are either marked by a green circle or left unmarked. More explicitly, in each layer of $T$, we mark a vertex if it is the $i$-th vertex, $i$ odd, when counting from the leftmost vertex of its layer. Its topmost vertex $v$ is always reachable. Indeed, if we set our initial state $S_0$ as the set of $2^k$ vertices in layer $k$, where $k \ge 1$, then we can jump with a peg on each unmarked vertex over its adjacent marked vertex so that $S_{2^{k-1}}$ consists of vertices in layer $k-1$. After $\sum_{i = 2}^k 2^{k-i}$ additional similar jumps, the game ends with a peg in $v$.

We claim that $d_n(v) = O(\vp^n)$. Write $D_n = D_n(v)$ and $d_n = d_n(v)$. Partition each set $D_k$ into two sets $M_k$ and $N_k$ consisting of marked and unmarked vertices, respectively. Set $m_k = |M_k|$ and $n_k = |N_k|$.

It is apparent from Figure \ref{fig2} that any given vertex in $M_k$ admits two adjacent vertices not already in $\bigcup_{n=0}^k D_n$, namely a vertex $a$ on its right and a vertex $b$ below it. We can see that $a$ belongs in $N_{k+1}$, while $b$ belongs in $M_{k+1}$. Likewise, any given vertex in $N_k$ admits an adjacent vertex not already in $\bigcup_{n=0}^k D_n$, namely a vertex below it belonging to $M_{k+1}$. This gives us
\[m_{k+1} = m_k + n_k = d_k\]
and
\[n_{k+1} = m_k = d_{k-1}.\]
Hence with initial values $d_0, d_1 = 1$, we have
\[d_{k+2} = m_{k+2} + n_{k+2} = d_{k+1} + d_k.\]
From this recurrence relation, $(d_n)$ is clearly the Fibonacci sequence, which is known to be in $O(\vp^n)$.
\end{proof}

\section{Limit states}\label{sec3}

Suppose that we have an infinite sequence of play $(S_n)_{n < \omega}$ on a graph $G$. We will construct the limit state $S_\omega$ as the set-theoretic limit of $(S_n)_{n < \omega}$. Write $\ol{S_n} = V(G) \setminus S_n$. Provided that every vertex is involved in finitely many jumps during the game, we define $S_\omega$ as the state consisting of vertices $v$ such that $v \in S_n$ for all sufficiently large $n$. Otherwise, $S_\omega$ is undefined. When $S_\omega$ is defined, its complement $\ol{S_\omega}$ consists of vertices $v$ with $v \in \ol{S_n}$ for all but finitely many $n$.

Our approach here is somewhat different to the one from Tatham and Taylor \cite{tatham}. According to their ruleset, the player makes a jump at instants of real-valued time. For example, the player may move at well-ordered times $0, \frac{1}{2}, \frac{2}{3}, \frac{3}{4}, \dots$ so that at time $t = 1$, the state of the game coincides with our formulation of $S_\omega$. However, the player may also move at non-well-ordered times $\dots, \frac{1}{4}, \frac{1}{3}, \frac{1}{2}, 1$, which means that there is no ``first move'' for that particular game when the initial state is recorded at $t =  0$. In our formulation, there is always a clear first jump, which produces a somewhat weaker game than that of Tatham and Taylor. However, since our set of states is well-ordered, the game becomes comparatively more straightforward.

We can use transfinite recursion to construct $S_\alpha$ for an arbitrary ordinal $\alpha$, whenever further play is possible. For every ordinal $\alpha$, the state $S_{\alpha + 1}$ can be obtained by performing a peg jump, denoted by $j_\alpha$, on $S_{\alpha}$. The state $S_\beta$, where $\beta$ is a limit ordinal, is the set of vertices $v$ such that $v \in S_{\lambda}$ for all sufficiently large $\lambda < \beta$, assuming that every vertex is involved in finitely many jumps. Focusing on countable graphs, we restrict our study to states indexed by a countable ordinal. The following theorem confirms that a game having states up to $S_\alpha$ with $\alpha$ countable can be simulated by a game with states up to $S_{\omega}$. Therefore, extending a game beyond $S_\omega$ is not a strict necessity.

\begin{theorem}\label{thm2}
Let $\alpha$ be a countable ordinal with $\alpha \ge \omega$, and suppose that $(S_\beta)_{\beta \le \alpha}$ is a sequence of defined game states on a graph $G$. There exists a sequence $(S'_n)_{n \le \omega}$ of game states on $G$ such that $S'_0 = S_0$ and $S'_\omega = S_{\alpha}$.
\end{theorem}

\begin{proof}
Two jumps are said to be \emph{disjoint} if there exists no vertex that is involved in both jumps; otherwise, they \emph{overlap}. For any fixed jump $j_\beta$ on $S_\beta$, where $\beta < \alpha$, we define
\[O_\beta = \{\lambda < \beta: j_\lambda \text{ and } j_\beta \text{ overlap}\}.\]
Take a sequence $(a_n)_{n < \omega}$, where $0 \le a_n < \alpha$ for every $n$, such that every ordinal $\beta < \alpha$ appears in the sequence infinitely many times; that is, $\beta = a_n$ for infinitely many $n$. This sequence is ensured to exist by the countability of $\alpha$.

We will construct a sequence $(j'_n)_{n < \omega}$ of jumps that encapsulates the original $(j_\beta)_{\beta < \alpha}$ by recursively defining a bijection $n \mapsto \beta_n$, assigning $j'_n = j_{\beta_n}$ in the process. Denote the sequence of game states produced from $(j'_n)_{n < \omega}$ by $(S')_{n \le \omega}$. For $n < \omega$, let
\[\beta_{<n} = \{\beta_m: m < n\} \quad \text{and} \quad T_n = \{a_n \le \xi < \alpha: \xi \notin \beta_{<n} \text{ and } O_\xi \subseteq \beta_{<n}\}.\]
Then define
\begin{equation}\label{eq3}
    \beta_n = \begin{cases}
    \min{T_n}, & \text{if } T_n \text{ is nonempty}, \\ \min{\{0 \le \xi < \alpha: \xi \notin \beta_{<n}\}}, & \text{if } T_n \text{ is empty}.
    \end{cases}
\end{equation}
Note that $O_{\beta_n} \subseteq \beta_{<n}$ for every $n < \omega$. In other words, jumps $j_\lambda$, where $\lambda < \beta_n$, that overlap with $j_{\beta_n}$ must be such that $\lambda = \beta_m$ for some $m < n$. We thus ensure that all previous jumps that overlap with $j_{\beta_n}$ have been played in the new sequence of jumps $(j'_k)_{k < n}$. This implies that for every $n$, the jump $j_{\beta_n}$ is a legal move to be played at state $S'_n$.

We claim that the map $n \mapsto \beta_n$ is a bijection. Since $\beta_n \notin \beta_{<n}$ for every $n$, we have $\beta_n \neq \beta_m$ for any $m < n$, so the map is injective. We now prove that it is surjective. Suppose to the contrary that there exists a least ordinal $\lambda$ not in the set $\beta_{<\omega} = \{\beta_n: n < \omega\}$. We have $O_\lambda \subseteq \beta_{<\omega}$ by the minimality of $\lambda$. Also, $O_\lambda$ must be finite, since otherwise $S_\lambda$ would be undefined. It follows that there exists some integer $m$ such that $O_\lambda \subseteq \beta_{<m}$. Taking any integer $i \ge m$ such that $a_i = \lambda$, we obtain from equation (\ref{eq3}) that $\beta_i = \lambda$, which contradicts our assumption that $\lambda \notin \beta_{<\omega}$.

We thus see that $(j'_n)$, where $j'_n = j_{\beta_n}$, is a valid sequence of jumps on $G$ which encapsulates the original sequence of jumps $(j_\beta)$. It is not hard to infer that $S'_\omega = S_\alpha$ given $S'_0 = S_0$, finishing our proof.
\end{proof}

\section{Valued graphs}\label{sec4}

Let $G$ be a connected, locally finite graph. We define for all $v \in V(G)$ the pagoda function $\pag_v(u) = \sigma^{d(u,v)}$, from which any state $S$ admits a value given by
\[\val_v(S) = \sum_{u \in S} \pag_v(u) = \sum_{u \in S} \sigma^{d(u,v)}\]
assuming the sum converges. This motivates the following class of graphs.

\begin{definition}
A connected, locally finite graph $G$ is \emph{valued} if for every $v \in V(G)$ and $S \subseteq V(G)$, the sum $\val_v(S) = \sum_{u \in S} \sigma^{d(u,v)}$ converges.
\end{definition}

To show that $G$ is a valued graph, one does not need to check every possible vertex $v$ and state $S$. This is a consequence of the following proposition.

\begin{proposition}\label{prop2}
A connected graph, locally finite $G = (V, E)$ is valued if there exists a vertex $w$ such that $\val_w(V) = \sum_{n = 0}^{\infty} d_n(w)\sigma^n$ converges.
\end{proposition}

\begin{proof}
Suppose that $v \in V$ is arbitrary and that $d(v,w) = k$. For $n \ge k$, we have the inequality
\[d_n(v) \le \sum_{i = -k}^k d_{n + i}(w).\]
Since $\sum_{n = 0}^{\infty} d_n(w)\sigma^n$ converges by assumption, $\sum_{n = 0}^{\infty} d_{n + i}(w)\sigma^{n+k}$ also converges for $0 \le i \le 2k$. Hence
\begin{align*}
\val_v(V) &= \sum_{n = 0}^{\infty} d_n(v)\sigma^n \\
&\le \sum_{n = 0}^{k-1} d_n(v)\sigma^n + \sum_{n = k}^{\infty} \sum_{i = -k}^k d_{n + i}(w)\sigma^{n} \\
&= \sum_{n = 0}^{k-1} d_n(v)\sigma^n + \sum_{i = 0}^{2k} \sum_{n = 0}^{\infty} d_{n + i}(w)\sigma^{n+k} \\
&< \infty.
\end{align*}
It directly follows that $\val_v(S)$ converges for every state $S$.
\end{proof}

Given a sequence $(S_n)_{n < \omega}$ of game states on valued graphs, $S_\omega$ is always defined and for every $v \in V(G)$, we have $\val_v(S_n) \to \val_v(S_\omega)$ as $n \to \infty$. To prove this, we first state a limit-sum interchange theorem known as Tannery's theorem. This is precisely Lebesgue's dominated convergence theorem for infinite series.

\begin{lemma}[\cite{bromwich1908}*{p. 123}]\label{lem3}
For every $n \in \N$, let $f_n \colon \N \to \R$ be a function, and suppose that the sequence $(f_n)$ converges pointwise. If there exists a function $g \colon \N \to \R$ such that $|f_n(m)| \le g(m)$ for $n, m \in \N$ and that $\sum_{m = 0}^{\infty} g(m)$ converges, then
\[\lim_{n \to \infty} \sum_{m = 0}^{\infty} f_n(m) = \sum_{m = 0}^\infty \lim_{n \to \infty} f_n(m).\]
\end{lemma}

\begin{theorem}\label{thm3}
If $(S_n)_{n < \omega}$ is a sequence of play on a valued graph $G$, then every vertex is involved in finitely many jumps during the game. Consequently, $S_\omega$ is defined. Furthermore,
\begin{equation}\label{eq4}
\val_v(S_\omega) = \lim_{n \to \infty} \val_v(S_n) = \inf_{n \in \N} \val_v(S_n)
\end{equation}
for all $v \in V(G)$.
\end{theorem}

\begin{proof}
Suppose that a vertex $v$ is involved in infinitely many jumps. Each time a peg on $v$ is brought elsewhere by some jump $j_k$, we have $\val_v(S_{k+1}) \le \val_v(S_k) - 1$. Consequently, the value of $S_n$ would be negative for sufficiently large $n$, which is impossible.

We now show that (\ref{eq4}) holds. For $n \le \omega$, define the function $h_n \colon V(G) \to \{0, 1\}$ as
\[h_n(u) = \begin{cases}
1, & \text{if } u \in S_n, \\ 0, & \text{if } u \in \ol{S_n}.
\end{cases}\]
For fixed $u$, note that $h_n(u) \to h_\omega(u)$ as $n \to \infty$. In addition,
\[\lim_{n \to \infty} \val_v(S_n) = \lim_{n \to \infty} \sum_{u \in S_n} \sigma^{d(u,v)} = \lim_{n \to \infty} \sum_{m = 0}^\infty \sum_{u \in D_m(v)} h_n(u)\sigma^m.\]
Setting $f_n(m) = \sum_{u \in D_m(v)} h_n(u)\sigma^m$ and $g(m) = d_m(v)\sigma^m$, we see that the hypotheses of Lemma \ref{lem3} are satisfied, which gives us
\begin{align*}
\lim_{n \to \infty} \sum_{m = 0}^\infty \sum_{u \in D_m(v)} h_n(u)\sigma^m &= \sum_{m = 0}^\infty \left(\lim_{n \to \infty} \sum_{u \in D_m(v)} h_n(u)\sigma^m\right)\\
&= \sum_{m = 0}^\infty \sum_{u \in D_m(v)} \left(\lim_{n \to \infty} h_n(u)\right)\sigma^m\\
&= \sum_{m = 0}^\infty \sum_{u \in D_m(v)} h_\omega(u)\sigma^m\\
&= \val_v(S_\omega).
\end{align*}
This proves the first equality of (\ref{eq4}); its second equality follows since $\val_v(S_n)$ is a non-increasing function of $n$.
\end{proof}

For a non-valued graph $G$, it is possible that $S_\omega = V(G)$ even though $S_0 \subsetneq V(G)$. The tree $T$ in Figure \ref{fig2} of Section \ref{sec2} provides such an example when we start with, say, $\ol{S_0} = \{v\}$. Indeed, if we start from $n = 1$ and then continue onward, layer $n$ in $T$ can be emptied to fill up a previously empty layer $n - 1$, and we will have $S_\omega = V(T)$. However, this can never be the case for valued graphs; it is impossible to fill up the whole graph with pegs when at least one vertex is pegless at the beginning of the game, even with infinitely many moves. This fact directly follows from the implication of Theorem \ref{thm3} that on valued graphs, $\val_v(S_\omega) \le \val_v(S_0)$ for any vertex $v$.

\begin{corollary}\label{cor1}
On a valued graph $G$, it is impossible to have $S_\omega = V(G)$ when $S_0 \subsetneq V(G)$.
\end{corollary}

Corollary \ref{cor1}, along with the fact that every vertex of a valued graph is eventually unreachable, shows that valued graphs form a class of well-behaved graphs in which there is a reasonable set of limitations on what the player is able to accomplish during a game.

\section{Infinite peg solitaire}\label{sec5}

In Section \ref{sec4}, we touched on the possibility of filling up an entire graph with pegs by state $S_\omega$. While this runs counter to the goal of peg solitaire, which is to empty the graph, it corresponds to \emph{fool's solitaire}, whose goal is to end the game with as many pegs on the board as possible when no more jumps can be made. For a selection of papers on fool's solitaire played on graphs, we refer to \cites{beeler2013, loeb2015}. This section and the next are devoted to the game of peg solitaire on infinite graphs.

In their treatment of peg solitaire on finite graphs, Beeler and Hoilman \cite{beeler2011} defined a \emph{solvable graph} as a graph $G$ on which there exists a vertex $v$ and a finite sequence of play $(S_n)_{n = 0}^k$ such that $\ol{S_0} = \{v\}$ and $|S_k| = 1$. If the vertex $v$ can be arbitrarily chosen from $G$, then $G$ is \emph{freely solvable}. For countably infinite graphs, we require an infinite sequence of play $(S_n)_{n \le \omega}$ so that solving the graph becomes attainable, where $S_\omega$ is a singleton in such a case. In addition, it is possible that a game ends with a completely pegless graph given an infinite number of jumps.

\begin{definition}
A countably infinite graph $G$ is \emph{clearable} if there exists a vertex $v$ and a sequence of play $(S_n)_{n \le \omega}$ such that $\ol{S_0} = \{v\}$ and $S_\omega = \emptyset$. The graph $G$ is \emph{freely clearable} if for every vertex $v$, there exists a sequence of play $(S_n)_{n \le \omega}$ such that $\ol{S_0} = \{v\}$ and $S_\omega = \emptyset$.
\end{definition}

Note that it is not necessary to directly end the game with a pegless graph $G$ by state $S_\omega$ to show that $G$ is clearable. As long as $\alpha \ge \omega$ is countable, a game on $G$ with $\left|\ol{S_0}\right| = 1$ and $S_\alpha$ empty implies that $G$ is clearable from Theorem \ref{thm2}.

\begin{figure}[t]
    \centering
    \begin{tikzpicture}[x=1cm, y=1cm]
        \draw[thick] (0, 0) -- (4, 0);
        \draw[thick] (6.5, 0) -- (10.5, 0);
        \draw[thick] (3.25, -2.3) -- (7.25, -2.3);
        \draw[dashed] (4, 0) -- (5, 0);
        \draw[dashed] (10.5, 0) -- (11.5, 0);
        \draw[dashed] (7.25, -2.3) -- (8.25, -2.3);
    
        \draw[fill=black] (0,0) circle (2pt);
        \draw[fill=white] (1,0) circle (2pt);
        \draw[fill=black] (2,0) circle (2pt);
        \draw[fill=black] (3,0) circle (2pt);
        \draw[fill=black] (4,0) circle (2pt);
        
        \draw[fill=black] (6.5,0) circle (2pt);
        \draw[fill=black] (7.5,0) circle (2pt);
        \draw[fill=white] (8.5,0) circle (2pt);
        \draw[fill=black] (9.5,0) circle (2pt);
        \draw[fill=black] (10.5,0) circle (2pt);
        
        \draw[fill=white] (3.25,-2.3) circle (2pt);
        \draw[fill=black] (4.25,-2.3) circle (2pt);
        \draw[fill=black] (5.25,-2.3) circle (2pt);
        \draw[fill=black] (6.25,-2.3) circle (2pt);
        \draw[fill=black] (7.25,-2.3) circle (2pt);
        
        \node at (0,-0.3) {$0$};
        \node at (1,-0.3) {$1$};
        \node at (2,-0.3) {$2$};
        \node at (3,-0.3) {$3$};
        \node at (4,-0.3) {$4$};
        
        \node at (6.5,-0.3) {$0$};
        \node at (7.5,-0.3) {$1$};
        \node at (8.5,-0.3) {$2$};
        \node at (9.5,-0.3) {$3$};
        \node at (10.5,-0.3) {$4$};
        
        \node at (3.25,-2.6) {$0$};
        \node at (4.25,-2.6) {$1$};
        \node at (5.25,-2.6) {$2$};
        \node at (6.25,-2.6) {$3$};
        \node at (7.25,-2.6) {$4$};
        
        \node at (2,-1) {(a) $\ol{S_0} = \{1\}$.};
        \node at (8.5,-1) {(b) $\ol{S_0} = \{2\}$.};
        \node at (5.25,-3.3) {(c) $\ol{S_0} = \{0\}$.};
    \end{tikzpicture}
    
    \caption{Three initial states on $P_\infty$.}
    \label{fig3}
\end{figure}
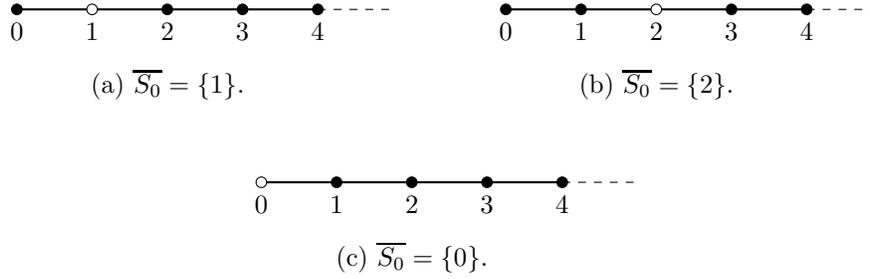

We define the \emph{ray} $P_\infty$ as a graph on $\N$ with edges of the form $\{n, n+1\}$. Likewise, the \emph{double ray} $P_{2\infty}$ is defined as a graph on $\Z$ with edges of the form $\{n, n+1\}$.

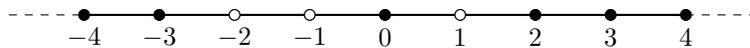
\begin{figure}[b]
    \centering
    \begin{tikzpicture}[x=1cm, y=1cm]
        \draw[thick] (-4, 0) -- (4, 0);
        \draw[dashed] (-5, 0) -- (-4, 0);
        \draw[dashed] (4, 0) -- (5, 0);
        
        \draw[fill=black] (-4,0) circle (2pt);
        \draw[fill=black] (-3,0) circle (2pt);
        \draw[fill=white] (-2,0) circle (2pt);
        \draw[fill=white] (-1,0) circle (2pt);
        \draw[fill=black] (0,0) circle (2pt);
        \draw[fill=white] (1,0) circle (2pt);
        \draw[fill=black] (2,0) circle (2pt);
        \draw[fill=black] (3,0) circle (2pt);
        \draw[fill=black] (4,0) circle (2pt);

        \draw[fill=black] (-4,0) circle (2pt);
        \draw[fill=black] (-3,0) circle (2pt);
        \draw[fill=white] (-2,0) circle (2pt);
        \draw[fill=white] (-1,0) circle (2pt);
        \draw[fill=black] (0,0) circle (2pt);
        \draw[fill=white] (1,0) circle (2pt);
        \draw[fill=black] (2,0) circle (2pt);
        \draw[fill=black] (3,0) circle (2pt);
        \draw[fill=black] (4,0) circle (2pt);

        \node at (-4,-0.3) {$-4$};
        \node at (-3,-0.3) {$-3$};
        \node at (-2,-0.3) {$-2$};
        \node at (-1,-0.3) {$-1$};
        \node at (0,-0.3) {$0$};
        \node at (1,-0.3) {$1$};
        \node at (2,-0.3) {$2$};
        \node at (3,-0.3) {$3$};
        \node at (4,-0.3) {$4$};
    \end{tikzpicture}
    
    \caption{$\ol{S_2} = \{-2, -1, 1\}$ on $P_{2\infty}$.}
    \label{fig4}
\end{figure}

\begin{proposition}\label{prop3}
The ray $P_\infty$ is both clearable and solvable, but not freely so.
\end{proposition}

\begin{proof}
Setting $\ol{S_0} = \{1\}$ as in Figure \ref{fig3}(a), we can designate $(2n + 1) \cdot \vv{2n} \cdot (2n - 1)$ as the $n$-th jump of the game, thus setting $S_\omega = \{n: n \text{ odd or zero}\}$. Finally, we have an empty state at $S_{\omega \cdot 2}$ by jumping the peg initially at $0$ two spaces ahead each turn, thus $P_\infty$ is clearable.

If we pick $\ol{S_0} = \{2\}$ as in Figure \ref{fig3}(b), then $\Z_+$ forms a \textit{subray} (that is, a ray that is also a subgraph) $R$ of $P_\infty$, with an initial state identical to Figure \ref{fig3}(a). Working exclusively on $R$, we can empty each vertex except $0$, thus $P_\infty$ is solvable.

On the other hand, if $\ol{S_0} = \{0\}$ as in Figure \ref{fig3}(c), then the only legal move for the $n$-th jump would be $(2n) \cdot \vv{2n - 1} \cdot (2n - 2)$. This leaves infinitely many vertices on $P_\infty$ at $S_\omega$. Consequently, $P_\infty$ is neither freely clearable nor freely solvable.
\end{proof}

\begin{proposition}\label{prop4}
The double ray $P_{2\infty}$ is neither solvable nor clearable.
\end{proposition}

\begin{proof}
Consider an infinite game on $P_{2\infty}$ with $\ol{S_0}$ a singleton. We prove that $S_\omega$ is infinite. The graph $P_{2\infty}$ at $S_2$ is, up to isomorphism, illustrated in Figure \ref{fig4}. If we never make the jump $3 \cdot \vv{2} \cdot 1$ after this state, then $S_\omega$ will be infinite. Hence suppose that $S_k$, where $k > 2$, is the state immediately after $3 \cdot \vv{2} \cdot 1$ is performed. We have $0, 1 \in S_k$. Notice that if neither $0 \cdot \vv{1} \cdot 2$ nor $1 \cdot \vv{0} \cdot (-1)$ is performed after $S_k$, then $S_\omega$ will be infinite. However, if $0 \cdot \vv{1} \cdot 2$ is performed, then $\Z_-$ necessarily contains infinitely many pegs at $S_\omega$. Likewise, $\Z_+$ will contain infinitely many pegs at $S_\omega$ if $1 \cdot \vv{0} \cdot (-1)$ is performed instead. In any case, $S_\omega$ must be infinite, so $P_{2\infty}$ is neither solvable nor clearable.
\end{proof}

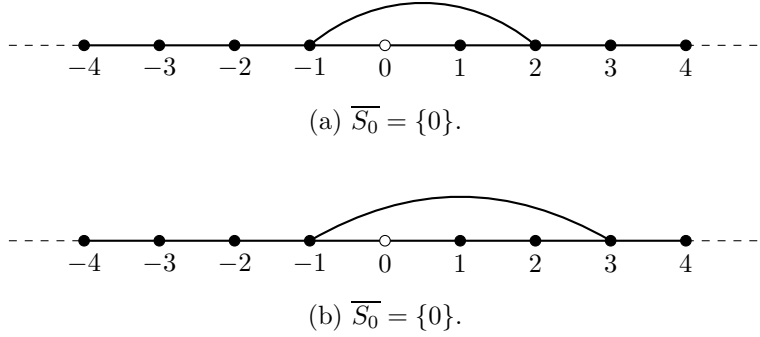
\begin{figure}[t]
    \centering
    \begin{tikzpicture}[x=1cm, y=1cm]
        \draw[thick] (-1, -2.6) to [bend left=40] (2, -2.6);
        \draw[thick] (-4, -2.6) -- (4, -2.6);
        \draw[dashed] (-5, -2.6) -- (-4, -2.6);
        \draw[dashed] (4, -2.6) -- (5, -2.6);
        
        \draw[thick] (-1, -5.2) to [bend left=30] (3, -5.2);
        \draw[thick] (-4, -5.2) -- (4, -5.2);
        \draw[dashed] (-5, -5.2) -- (-4, -5.2);
        \draw[dashed] (4, -5.2) -- (5, -5.2);

        \draw[fill=black] (-4,-2.6) circle (2pt);
        \draw[fill=black] (-3,-2.6) circle (2pt);
        \draw[fill=black] (-2,-2.6) circle (2pt);
        \draw[fill=black] (-1,-2.6) circle (2pt);
        \draw[fill=white] (0,-2.6) circle (2pt);
        \draw[fill=black] (1,-2.6) circle (2pt);
        \draw[fill=black] (2,-2.6) circle (2pt);
        \draw[fill=black] (3,-2.6) circle (2pt);
        \draw[fill=black] (4,-2.6) circle (2pt);
        
        \draw[fill=black] (-4,-5.2) circle (2pt);
        \draw[fill=black] (-3,-5.2) circle (2pt);
        \draw[fill=black] (-2,-5.2) circle (2pt);
        \draw[fill=black] (-1,-5.2) circle (2pt);
        \draw[fill=white] (0,-5.2) circle (2pt);
        \draw[fill=black] (1,-5.2) circle (2pt);
        \draw[fill=black] (2,-5.2) circle (2pt);
        \draw[fill=black] (3,-5.2) circle (2pt);
        \draw[fill=black] (4,-5.2) circle (2pt);

        \node at (-4,-2.9) {$-4$};
        \node at (-3,-2.9) {$-3$};
        \node at (-2,-2.9) {$-2$};
        \node at (-1,-2.9) {$-1$};
        \node at (0,-2.9) {$0$};
        \node at (1,-2.9) {$1$};
        \node at (2,-2.9) {$2$};
        \node at (3,-2.9) {$3$};
        \node at (4,-2.9) {$4$};

        \node at (-4,-5.5) {$-4$};
        \node at (-3,-5.5) {$-3$};
        \node at (-2,-5.5) {$-2$};
        \node at (-1,-5.5) {$-1$};
        \node at (0,-5.5) {$0$};
        \node at (1,-5.5) {$1$};
        \node at (2,-5.5) {$2$};
        \node at (3,-5.5) {$3$};
        \node at (4,-5.5) {$4$};
        
        \node at (0,-3.6) {(a) $\ol{S_0} = \{0\}$.};
        \node at (0,-6.2) {(b) $\ol{S_0} = \{0\}$.};
    \end{tikzpicture}
    
    \caption{Game start on a couple of two-ended graphs on the integers.}
    \label{fig5}
\end{figure}

By inserting the edge $\{-1, 2\}$ to $P_{2\infty}$, we can obtain a clearable graph. Indeed, starting with the configuration displayed in Figure \ref{fig5}(a), we can clear the graph by performing the jump $2 \cdot \vv{-1} \cdot 0$ before emptying $\Z_+$ and $\Z_{\le 0}$ independently. From this observation, we can obtain a freely clearable graph by adding edges of the form $\{n, n+3\}$ to $P_{2\infty}$. Likewise, by inserting the edge $\{-1, 3\}$ to $P_{2\infty}$, we obtain a solvable graph since we can clear all vertices other than $1$. Therefore, adding edges of the form $\{n, n+4\}$ to $P_{2\infty}$ results in a freely solvable graph. We remark that the concept of adding edges to obtain solvable graphs was considered in detail for finite graphs in \cites{beeler2016, de2020, de2022}.

Theorem \ref{thm4} implies that an infinite rayless graph is neither solvable nor clearable. The proof makes use of the K\H{o}nig's infinity lemma (Lemma \ref{lem4}) and a simple result on \textit{one-way infinite walks}, which are defined as a sequence of vertices $w_0w_1 \dots$ where $w_i$ is adjacent to $w_{i-1}$ for all $i \ge 1$.

\begin{lemma}[\cite{diestel2017}*{Lemma 8.1.2}]\label{lem4}
Let $(V_n)_{n = 0}^\infty$ be an infinite sequence of disjoint nonempty finite sets, and let $H$ be a graph on $\bigcup_{n = 0}^\infty V_n$. If every vertex in $V_{n+1}$ has a neighbor in $V_n$, then $H$ contains a ray $v_0v_1\dots$ such that $v_n \in V_n$ for $n \in \N$.
\end{lemma}

\begin{lemma}\label{lem5}
Let $W = w_0w_1\dots$ be a one-way infinite walk on a graph $G$. If no vertex is repeated infinitely many times in $W$, then $G$ contains a subray.
\end{lemma}

\begin{proof}
We can construct a ray $r_0r_1\dots$ on $G$ with $r_0 = w_0$ and $r_{n+1} = w_{a_n}$ for $n \in \N$, where
\[a_n = \max_{w_k = r_n} k + 1.\qedhere\]
\end{proof}

\begin{theorem}\label{thm4}
If $(S_n)_{n < \omega}$ is a sequence of game states on an infinite rayless graph $G$, with $\ol{S_0}$ a nonempty finite set, then $S_\omega$ is undefined.
\end{theorem}

\begin{proof}
Let $(j_n)_{n < \omega}$ be the corresponding sequence of jumps. For every $n \in \N$, define
\begin{equation}\label{eq5}
    V_n = \{v \in \ol{S_n}: v \text{ is involved in } j_i \text{ for some } i \ge n\}.
\end{equation}
For $n \in \N$, the set $V_n$ is nonempty since the jump $j_n$ in particular involves a vertex in $\ol{S_n}$. We also have $\left|V_n\right| \le \left|\ol{S_n}\right| = \left|\ol{S_0}\right| + n$. Hence each $V_n$ is nonempty and finite. Here we treat vertices in sets $V_k$ and $V_\ell$, where $k \neq \ell$, as different objects so that $(V_n)_{n = 0}^\infty$ consists of disjoint sets. This ensures that $(V_n)_{n = 0}^\infty$ satisfies the hypothesis of Lemma \ref{lem4}.

Let $H$ be a graph on $\bigcup_{n = 0}^\infty V_n$ on which an edge is inserted between $v_{n+1} \in V_{n+1}$ and $v_n \in V_n$ whenever one of the following holds:
\begin{enumerate}[(i)]
    \item $v_{n+1}$ is involved in $j_n$, and $v_n$ is the unique vertex in $V_n$ that is also involved in $j_n$.
    \item $v_{n+1}$ is not involved in $j_n$, and $v_n$ denotes the same vertex as $v_{n+1}$.
\end{enumerate}
From this construction, every vertex in $V_{n+1}$ has a unique neighbor in $V_n$.

By Lemma \ref{lem4}, there exists a ray $v_0v_1\dots$ on $H$ such that $v_n \in V_n$ for $n \in \N$. Notice that every pair of vertices $v_n$ and $v_{n+1}$ are either the same vertex, adjacent, or distance two apart in $G$. If they are distance two apart, then $j_n = v_{n+1} \cdot \vv{w} \cdot v_n$ for some unique vertex $w$, so $w$ can be inserted between $v_n$ and $v_{n+1}$ inside the ray. If $v_n = v_{n+1}$, then one of them can be deleted from the ray. After performing all the necessary insertions and deletions, we obtain a walk $W$ on $G$.

The walk $W$ is infinite since if $W$ ended on a vertex $v_k$, then $v_k$ would not be involved in $j_i$ for every $i \ge k$, which contradicts the definition of $V_k$ in (\ref{eq5}). Suppose toward a contradiction that $S_\omega$ is defined, which implies that every vertex is involved in only finitely many jumps. It follows that no vertex is repeated infinitely many times in $W$. Hence by Lemma \ref{lem5}, there exists a subray on $G$, contradicting our assumption that $G$ is rayless.
\end{proof}

\section{Clearability of Cartesian products}\label{sec6}

Recall that the \textit{Cartesian product} $G \cp H$ of two graphs $G$ and $H$ is a graph on $V(G) \times V(H)$ where $\{(g_1, h_1), (g_2, h_2)\} \in E(G \cp H)$ if and only if either $g_1 = g_2$ and $h_1h_2 \in E(H)$, or $h_1 = h_2$ and $g_1g_2 \in E(G)$. Given $h \in V(H)$, the subgraph of $G \cp H$ isomorphic to $G$ induced by vertices of the form $(-, h)$ is denoted by $G_h$. Likewise, the subgraph induced by vertices $(g, -)$ for some fixed $g \in V(G)$ is denoted by $H_g$.

In \cites{beeler2011, kreh2021, loeb2015}, several solvability notions on finite Cartesian products are examined. In this last section, we instead study several clearability conditions for infinite Cartesian products, starting with the straightforward fact that the Cartesian product of two clearable graphs is clearable. We also show that $G \cp P_2$ is clearable whenever $G$ is clearable, where $P_2$ is the $2$-vertex path.

\begin{proposition}\label{prop5}
If $G$ is a clearable graph and $H$ is either a $P_2$ or a clearable graph, then $G \cp H$ is clearable.
\end{proposition}

\begin{proof}
Assume that $G$ is a clearable graph with $g_0 \in V(G)$ taken as the initial empty vertex. Suppose that $V(P_2) = \{0, 1\}$. Taking $g_1$ as a vertex of $G$ adjacent to $g_0$, set $\ol{S_0} = \{(g_1, 0)\}$. Now perform $(g_0, 1) \cdot \vv{(g_0, 0)} \cdot (g_1, 0)$ so that $\ol{S_1} = \{(g_0, 0), (g_0, 1)\}$. We can proceed to empty $G \cp P_2$ by clearing $G_0$ and $G_1$ independently.

Now suppose that $H$ is clearable with $h_0 \in V(G)$ as the initial empty vertex. Set $\ol{S_0} = \{(g_0, h_0)\}$. The graph $G \cp H$ can be cleared by first clearing $H_{g_0}$ so that $\ol{S_\omega} = \{(g_0, h): h \in V(H)\}$, before independently clearing every subgraph $G_h$ for $h \in V(H)$.
\end{proof}

A direct consequence of Propositions \ref{prop3} and \ref{prop5} is that the grid $\N^k$ is clearable for $k \ge 2$. As our final result, we prove that the Cartesian product of a freely clearable graph and any countable, connected, locally finite graph is freely clearable.

\begin{theorem}\label{prop6}
If $G$ is a freely clearable graph and $H$ is a countable, connected, locally finite graph, then $G \cp H$ is freely clearable.
\end{theorem}

\begin{proof}
Suppose that $|V(H)| = n$, where $n \le \omega$. We say that an enumeration $(h_k)_{0 \le k < n}$ of $V(H)$ is a \emph{traversal of $H$} if $H[h_0, \dots, h_k]$---the subgraph of $H$ induced by the vertices $h_0, \dots, h_k$---is a connected subgraph for $0 \le k < n$. For any $h_0 \in V(H)$, it can be shown that $H$ admits a traversal with $h_0$ as its starting vertex (see \cite{bhaskar2018} for an example of the construction).

Given $\ol{S_0} = \{(g_0, h_0)\}$, we will provide a way to empty $G \cp H$. Pick any vertex $g_1$ adjacent to $g_0$ in $G$, and let $(h_k)_{k < n}$ be a traversal of $H$. For $1 \le k < n$, take any integer $a_k < k$ such that $h_{a_k}$ is adjacent to $h_k$, and then define the jump
\[j_{k - 1} = \begin{cases}
(g_1, h_k) \cdot \vv{(g_1, h_{a_k})} \cdot (g_0, h_{a_k}), & \text{if } (g_0, h_{a_k}) \text{ is pegless}, \\ (g_0, h_k) \cdot \vv{(g_0, h_{a_k})} \cdot (g_1, h_{a_k}), & \text{if } (g_1, h_{a_k}) \text{ is pegless}.
\end{cases}\]
This construction is well-defined since for $0 \le m \le k < n$, the set $\ol{S_k} \cap V(G_{h_m})$ is either equal to $\{(g_0, h_m)\}$ or $\{(g_1, h_m)\}$.

Let $S$ be the resulting state after the sequence of jumps $(j_k)$ is performed. We claim that $S$ is defined and that $\left|\ol{S} \cap V(G_h)\right| = 1$ for all $h \in V(H)$. By local finiteness, no integer can be repeated infinitely many times in the sequence $(a_k)$. Hence every vertex is involved in finitely many jumps, and $S$ is thus defined. Also, we know that $\left|\ol{S_k} \cap V(G_{h_m})\right| = 1$ for $0 \le m \le k < n$. Consequently, we have for every $h \in V(H)$ and all sufficiently large $k$ (depending on $h$) that $\left|\ol{S_k} \cap V(G_h)\right| = 1$, and $\ol{S_k} \cap V(G_h)$ is either the set $\{(g_0, h)\}$ or $\{(g_1, h)\}$. Thus $\left|\ol{S} \cap V(G_h)\right| = 1$ as required. Continuing on from the state $S$, we can independently clear each $G_h$ for $h \in V(H)$, hence emptying the graph.
\end{proof}

\end{document}